\documentclass[oneside,english]{amsart}
\usepackage[T1]{fontenc}
\usepackage[latin9]{inputenc}
\usepackage{amsthm}
\pdfoutput=1 
\usepackage{graphicx}
\usepackage[authoryear,round]{natbib}

\makeatletter

\providecommand{\tabularnewline}{\\}
\numberwithin{equation}{section}
\numberwithin{figure}{section}
\theoremstyle{plain}
\newtheorem{thm}{\protect\theoremname}
\theoremstyle{plain}
\newtheorem{cor}[thm]{\protect\corollaryname}
\theoremstyle{definition}
\newtheorem{example}[thm]{\protect\examplename}

\usepackage{babel}

\makeatother

\usepackage{babel}
\providecommand{\corollaryname}{Corollary}
\providecommand{\examplename}{Example}
\providecommand{\theoremname}{Theorem}

\begin{document}
\title{Enhancing Markov and Chebyshev's inequalities}
\author{Joan del Castillo}
\address{Departament de Matemàtiques, Universitat Autònoma de Barcelona, Cerdanyola
del Vallès, Spain.}
\email{joan.delcastillo@uab.cat}
\keywords{Restricted mean, Inequalities in Probability, Concentration Inequalities.}
\begin{abstract}
The idea of the restricted mean has been used to establish a significantly
improved version of Markov's inequality that does not require any
new assumptions. The result immediately extends on Chebyshev's inequalities
and Chernoff's bound. The improved Markov inequality yields a bound
that is hundreds or thousands of times more accurate than the original
Markov bound for high quantiles in the most prevalent and diverse
situations. 

The Markov inequality benefits from being model-independent, and the
long-standing issue of its imprecision is solved. Practically speaking,
avoidance of model risk is decisive when multiple competing models
are present in a real-world situation. 
\end{abstract}

\maketitle

\section{Introduction}

Historically, Markov and Chebyshev's inequalities, which provide proof
of the weak law of large numbers, date back to the articles by \citet{tchb67}
and \citet{mark84}. Today's most basic version of this inequality
carries Markov's name. Chebyshev was undoubtedly aware of it as well,
and Markov was a student of his. These inequalities have survived
to our day in most undergraduate probability textbooks, for instance,
\citet{ross14} and \citet{case02}.

Markov's inequality has the benefit of being model-independent from
a practical standpoint. When numerous competing models are provided
in relation to a real-world problem, the elimination of model risk
is crucial. The first problem with inequality is its lack of accuracy.
Recently, several researchers have taken an interest in the idea of
updating these classic inequalities in relation to current problems,
see \citet{saty14}, \citet{cohe15}, \citet{hube19}, \citet{ogas20},
\citet{ren21}, \citet{sola22}, \citet{bhat22}, \citet{choi23}.

Theorem \ref{thm:1} provides an enhanced Markov inequality that is
hundreds or thousands of times more accurate than the original for
high quantiles and does not require any new assumptions, see Figure
\ref{fig:F1} and Table \ref{tab:T1}. The result immediately extends
Chebyshev's inequalities and Chernoff's bound \citep{cher52}, see
Corollary \ref{cor C2}. The proof is practically as simple as that
of the traditional Markov inequality, using the concept of restricted
expectation. This idea is used, for example, in the Lorenz curve \citep{lore05},
which measures how wealth is distributed in a society, see \citet{arno15},
and in medicine as the restricted mean survival time, see \citet{roys13}. 

\section{\label{sec:Improving-on-the S1}Improving the accuracy of Markov's
inequality}

According to Markov's inequality, see \citet{ross14}, the probability
of a nonnegative random variable $X$ being at least $\nu>0$ is at
most the expectation of $X$ divided by $\nu$. That is 
\begin{equation}
\Pr\left\{ X>\nu\right\} \leq E\left(X\right)/\nu.\label{eq:Mark}
\end{equation}
Chebyshev's inequality, an extended version for a nondecreasing nonnegative
function, $\varphi$, is as follows 
\[
\Pr\left\{ X>\nu\right\} \leq E(\varphi\left(X\right))/\varphi\left(\nu\right),
\]
The proof of this extended version is as simple as applying Markov
(\ref{eq:Mark}) to the second part of the equality: 
\begin{equation}
\Pr\left\{ X>\nu\right\} =\Pr\left\{ \varphi\left(X\right)>\varphi\left(\nu\right)\right\} .\label{eq:fi}
\end{equation}
The proof of the weak law of large numbers follows from Chebyshev's
inequality for the variable$\left|X-E\left(X\right)\right|$ and $\varphi\left(x\right)=x^{2}$,
when $X$ has a finite variance, see \citet{case02}. Other immediate
inequalities follow using higher moments, $\varphi\left(X\right)=X^{k}$,
and using the moment-generating function, $\varphi\left(X\right)=exp\left(tX\right)$,
which led to Chernoff's bound (1952).

A typical proof of Markov's inequality makes use of the indicator
random variable of the event $\left\{ X>\nu\right\} $ 
\[
1_{\left\{ X>\nu\right\} }=\left\{ \begin{array}{cc}
0,\;when & X\leq\nu\\
1,\;when & X>\nu
\end{array}\right.
\]
and the inequality 
\begin{equation}
\nu1_{\left\{ X>\nu\right\} }\leq X\label{eq:indic}
\end{equation}
which is clear if we consider the two possible values of \emph{X.}
If $0\leq X\leq\nu$, $\nu1_{\left\{ X>\nu\right\} }=0$ and if $X>\nu$,
$\nu1_{\left\{ X>\nu\right\} }=\nu$. Then, taking the expectation
of both sides of the inequality (\ref{eq:indic}), Markov inequality
(\ref{eq:Mark}) follows, since $E\left(1_{\left\{ X>\nu\right\} }\right)=\Pr\left\{ X>\nu\right\} $
and expectation is a monotonically increasing operator.

Let $X$ be a nonnegative random variable with cumulative distribution
function $F\left(x\right)$. For clarity, it is now assumed that $X$
has a probability density function $f\left(x\right)$, so its expectation
is expressed as 
\[
E\left(X\right)=\int_{0}^{\infty}x\,f(x)\,dx,
\]

For $\nu\geq0$, the \emph{restricted expectation }of $X$ over $\nu$
is introduced as follows 
\begin{equation}
E_{\nu}\left(X\right)=E\left(X1_{\left\{ X>\nu\right\} }\right)=\int_{\nu}^{\infty}x\,f(x)\,dx.\label{eq:restrE}
\end{equation}
Naturally, if $\nu=0$, $E_{0}\left(X\right)=E\left(X\right)$. This
concept is employed, for instance, in the Lorenz curve \citep{arno15},
and in medicine as the restricted mean survival time \citep{roys13}.

The following is the paper's primary finding. 
\begin{thm}
\label{thm:1}If $X$ is a nonnegative random variable and $\nu>0$,
then the probability of $X$ being at least $\nu$ is at most the
restricted expectation of X over $\nu$ divided by $\nu$. This improves
on Markov's inequality. 
\[
\Pr\left\{ X>\nu\right\} \leq E_{\nu}\left(X\right)/\nu\leq E\left(X\right)/\nu.
\]
\end{thm}

\begin{proof}
Multiplying both sides of the inequality (\ref{eq:indic}) by the
nonnegative random variable $1_{\left\{ X>\nu\right\} }$, bearing
in mind that $\left(1_{\left\{ X>\nu\right\} }\right)^{2}=1_{\left\{ X>\nu\right\} }$,
the following inequality follows 
\begin{equation}
\nu1_{\left\{ X>\nu\right\} }\leq X1_{\left\{ X>\nu\right\} },\label{eq:ine2}
\end{equation}
Then taking the expectation of both sides of the above inequality,
the main part of the result follows. Moreover, it is clear that $X1_{\left\{ X>\nu\right\} }\leq X$
and taking the expectation on both sides it is demonstrated that the
new inequality improves on Markov's inequality.

Note that with a measure-theoretic definition of expectation, the
probability density function is not necessary and the result is as
general as traditional Markov inequality. In particular for discrete
random variables, this conclusion is true. 
\end{proof}
The three terms of the inequalities of Theorem \ref{thm:1} will henceforth
be called the tail function\emph{, }$\Pr\left\{ X>\nu\right\} =(1-F(x))=\overline{F}\left(x\right)$\emph{,}
traditional Markov bound, $E\left(X\right)/\nu$, and enhanced Markov
bound, $E_{\nu}\left(X\right)/\nu$. Enhanced Markov inequality will
from now on be used to refer to the first inequality of Theorem \ref{thm:1}.

Theorem \ref{thm:1} for Markov's inequality instantly generalizes
Chebyshev and Chernoff's inequalities under the restricted expectation,
as the following result shows. 
\begin{cor}
\label{cor C2}If $X$ is a nonnegative random variable and $\nu>0$,
then for a nondecreasing nonnegative function $\varphi$ 
\[
\Pr\left\{ X>\nu\right\} \leq E_{\nu}(\varphi\left(X\right))/\varphi\left(\nu\right)\leq E(\varphi\left(X\right))/\varphi\left(\nu\right),
\]
\end{cor}

\begin{proof}
We simply need to apply Theorem \ref{thm:1} to the variable $\varphi\left(X\right)$
in the second part of equality (\ref{eq:fi}). 
\end{proof}
For instance, using higher moments, with $\varphi\left(X\right)=X^{k}$,
the previous Corollary leads to the following conclusion. 
\begin{equation}
\Pr\left\{ X>\nu\right\} \leq E_{\nu}(X^{k})/\nu^{k}\leq E(X^{k})/\nu^{k},\label{eq:momM}
\end{equation}
Using the moment generating function, with $\varphi\left(X\right)=exp\left(tX\right)$,
the Chernoff bound is also improved 
\begin{equation}
\Pr\left\{ X>\nu\right\} \leq E_{\nu}(exp\left(tX\right))e^{-t\nu}\leq E(exp\left(tX\right))e^{-t\nu}\label{eq:Chern}
\end{equation}

Below, the results of Theorem \ref{thm:1} will be checked in different
examples to show the extraordinary improvement they represent. It
will also be shown how the restricted expectation can be computed
just as easily as the general expectation. We believe that the large
increase in accuracy will lead to significant practical applications.

\section{Examples}

In order to confirm the assertion expressed in the Introduction, we
offer some examples to go along with the Section \ref{sec:Improving-on-the S1}
results. For the half-normal and exponential distributions across
high quantiles, we will compare the enhanced and traditional Markov
bounds. 

In general, for a sample of size $n$, we shall take into account
the $(1-1/n)$ quantile, which is typically a little below the sample
maximum. The improved Markov bound is often about $n$ times more
accurate than the traditional bound when we consider samples larger
than $100$.

Figure \ref{fig:F1} shows a preliminary idea of how the enhanced
Markov bound and the traditional Markov bound approximate the tail
function in the half-normal and exponential distributions with expected
value 1.

\begin{figure}
\caption{\label{fig:F1}The tail function (thick), the enhanced Markov bound
(thin) for half-normal (blue) and exponential (red) distributions,
both with expected value 1. The traditional Markov bound (dashed,
blue) is the same in both cases.}

\includegraphics{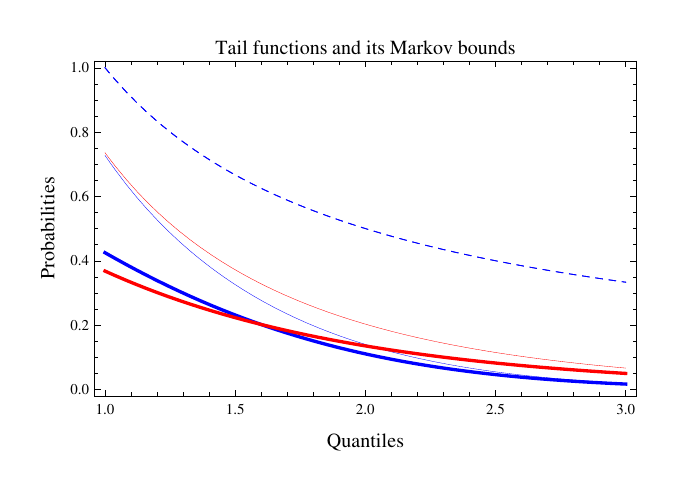} 
\end{figure}

\begin{example}
The probability density function of the standard exponential distribution,
$X$, is 
\[
f\left(x\right)=e^{-x},\;x\geq0.
\]
The cumulative distribution function is given by 
\[
F\left(x\right)=1-e^{-x},\;x\geq0.
\]
The restricted expectation is straightforward to compute using integration
by parts 
\[
E_{\nu}\left(X\right)=\int_{\nu}^{\infty}x\,e^{-x}\,dx=e^{-\nu}\left(1+\nu\right),
\]
and if $\nu=0$, the expectation is $E\left(X\right)=1$.

The $0.99$ and $0.999$ quantiles for $X$ are $4.605$ and $6.908$.
The three terms of the inequalities of Theorem \ref{thm:1} for $\nu=6.908$
are: the tail function $1.00\,10^{-3}$, the enhanced\emph{ }Markov
bound $1.14\,10^{-3}$ and the traditional Markov bound $0.145$.
Simply by looking at the orders of magnitude, the difference is obvious,
for it is about a thousand times better. See also Table \ref{tab:T1}. 
\end{example}

\begin{table}
\caption{\label{tab:T1}Comparison between the adjustments of the tail function
by the enhanced Markov bound, $E_{\nu}\left(X\right)/\nu$, and traditional
Markov bound, $E\left(X\right)/\nu$, for the half-normal and exponential
distributions with expected value $1$.}

\begin{tabular}{|c|c|c|c|c|c|c|}
\hline 
 & \multicolumn{3}{c|}{Half normal} & \multicolumn{3}{c|}{Exponential}\tabularnewline
\hline 
\hline 
$\nu$ & $1-F(\nu)$ & $E_{\nu}\left(X\right)/\nu$ & $E\left(X\right)/\nu$ & $1-F(\nu)$ & $E_{\nu}\left(X\right)/\nu$ & $E\left(X\right)/\nu$\tabularnewline
\hline 
1 & 0.425 & 0.727 & 1.000 & 0.368 & 0.736 & 1.000\tabularnewline
\hline 
2 & 0.111 & 0.140 & 0.500 & 0.135 & 0.203 & 0.500\tabularnewline
\hline 
3 & 0.017 & 0.019 & 0.333 & 0.050 & 0.066 & 0.333\tabularnewline
\hline 
4 & 1.4E-03 & 1.5E-03 & 0.250 & 0.018 & 0.023 & 0.250\tabularnewline
\hline 
5 & 6.6E-05 & 7.0E-05 & 0.200 & 6.7E-03 & 8.1E-03 & 0.200\tabularnewline
\hline 
6 & 1.7E-06 & 1.8E-06 & 0.167 & 2.5E-03 & 2.9E-03 & 0.167\tabularnewline
\hline 
7 & 2.3E-08 & 2.4E-08 & 0.143 & 9.1E-04 & 1.0E-03 & 0.143\tabularnewline
\hline 
8 & 1.7E-10 & 1.8E-10 & 0.125 & 3.4E-04 & 3.8E-04 & 0.125\tabularnewline
\hline 
\end{tabular}
\end{table}

If $X$ is normally distributed with zero mean, $N\left(0,\sigma^{2}\right)$,
then its absolute value$\left|X\right|$ follows a \emph{half-normal}
distribution with a probability density function 
\[
f_{H}\left(x;\sigma\right)=\frac{2}{\sqrt{2\pi}\sigma}exp\left(-\frac{x^{2}}{2\sigma^{2}}\right),\;x\geq0.
\]
The distribution is supported on the interval $[0,\infty)$. If $\Phi\left(x\right)$
is the cumulative distribution function of the standard normal distribution,
then the cumulative distribution function of the \emph{half-normal}
distribution is 
\[
F_{H}\left(x;\sigma\right)=2\Phi\left(x/\sigma\right)-1,\;x\geq0.
\]

\begin{example}
The restricted expectation of the half-normal distribution is given
by 
\[
E_{\nu}\left(\left|X\right|\right)=\frac{\sqrt{2}}{\sqrt{\pi}\sigma}\int_{\nu}^{\infty}x\,exp\left(-\frac{x^{2}}{2\sigma^{2}}\right)\,dx,
\]
since 
\[
-\sigma^{2}\frac{d}{dx}\left\{ exp\left(-\frac{x^{2}}{2\sigma^{2}}\right)\right\} =x\,exp\left(-\frac{x^{2}}{2\sigma^{2}}\right),
\]
then 
\[
E_{\nu}\left(\left|X\right|\right)=\sigma\sqrt{2/\pi}exp\left(-\frac{\nu^{2}}{2\sigma^{2}}\right),
\]
in particular $E\left(\left|X\right|\right)=\sigma\sqrt{2/\pi}$.

Then, a half-normal distribution $Z$ where $\sigma=\sqrt{\pi/2}$
has unit expectation, $E\left(Z\right)=1$, and its restricted expectation
is 
\[
E_{\nu}\left(Z\right)=exp\left(-\nu^{2}/\pi\right),
\]

The $0.99$ and $0.999$ quantiles for $Z$ are $3.228$ and $4.124$.
The three terms of the inequalities of Theorem \ref{thm:1} for $Z$
and $\nu=4.124$ are: the tail function $1.00\,10^{-3}$, the enhanced\emph{
}Markov bound $1.08\,10^{-3}$ and the traditional Markov bound $0.242$.
The new inequality is thousands times better, see also Table \ref{tab:T1}. 
\end{example}

Table \ref{tab:T1} compares the approximations of the tail function
by the enhanced and traditional Markov bounds for the half-normal
and exponential distributions with expected value $1$. The two functions
are evaluated up to $\nu=8$, which is far from the expected value.
The comparison between tail function fits does not depend on the scale
parameter. The Table shows that the enhanced Markov bound follows
the tail function very closely, unlike the traditional Markov bound.
The approximations of the enhanced Markov bound follow the tail function
in all orders of magnitude, up to $10^{-10},$ while the traditional
Markov is thousands of times less accurate in both cases for samples
larger than $10^{3}$.

\bibliography{markov}

\bibliographystyle{plainnat}
\end{document}